\documentclass{article}
\usepackage[%
journal=JST,    %% Replace XXX by one of the above journal codes.
lang=american,   %% Change to `american' if you use American English.
]{ems-journal}

\newtheorem{theorem}{Theorem}
\newtheorem{lemma}{Lemma}
\newtheorem{proposition}{Proposition}
\newtheorem{definition}{Definition}
\newtheorem{remark}{Remark}

\numberwithin{equation}{section}

\begin{document}

%------
% Insert the title of your paper and (if necessary)
% a short title for the running head.
%------
\title{Equality in some symplectic eigenvalue inequalities}
\titlemark{Equality in some symplectic eigenvalue inequalities}

%------

%%%% Pls fill in all fields for each author
%%%% Label the authors by their position in the authors' list using {}
%%%% If you published any math paper ever, you have an MR Author ID.
%  Please look it up in three easy (and free) steps:
% 1. copy the bibliographic data of any published paper (co-)authored by you in the search field at https://mathscinet.ams.org/mathscinet/freetools/mref
% 2. Hit your name in the search result
% 3. Find your MR Author ID in the first row, copy it in the \mrid{} field
%%%% If you have not created your ORCID yet, you may like to do it now, pls copy it in the field \orcid{}
%%%% Abbreviate first names for the running head

\emsauthor{1}{
	\givenname{Hemant K.}
	\surname{Mishra}
	\mrid{1392358}
	\orcid{0000-0001-9823-9025}}{Hemant K.~Mishra}
%%%% Repeat the same fields for each numbered author
% \emsauthor{2}{
% 	\givenname{Second}
% 	\surname{Contributor}
% 	\mrid{}
% 	\orcid{}}{S.~Contributor}
% \emsauthor{3}{
% 	\givenname{Someone}
% 	\surname{Else}
% 	\mrid{}
% 	\orcid{}}{S.~Else}

%%%% Please provide detailed address info for each author
%%%% Use the same numbering as for \emsauthor above
%%%% Please look up the ROR ID of your institute here: https://ror.org
\Emsaffil{1}{
	\department{School of Electrical and Computer Engineering}
	\organisation{Cornell University}
	\rorid{05bnh6r87}
	\address{}
	\zip{14850}
	\city{Ithaca}
	\country{USA}
	\affemail{hemant.mishra@cornell.edu}}
%%%% Repeat the same fields for each numbered author
%%%% If some author has multiple affiliations, repeat the fields for each affiliation
%%%% Number the affiliations using {}
% \Emsaffil{2}{
% 	\department{1}{}
% 	\organisation{1}{}
% 	\rorid{1}{}
% 	\address{1}{}
% 	\zip{1}{}
% 	\city{1}{}
% 	\country{1}{}
% %
% 	\department{2}{} 
% 	\organisation{2}{}%
% 	\rorid{2}{}
% 	\address{2}{}%
% 	\zip{2}{}
% 	\city{2}{}
% 	\country{2}{} 
% 	\affemail{2}{}}
% \Emsaffil{3}{
% 	\department{}
% 	\organisation{}
% 	\rorid{}
% 	\address{}
% 	\zip{}
% 	\city{}
% 	\country{}
% 	\affemail{}}

%------
% Add MSC 2020 codes according to https://zbmath.org/classification/.
% A unique primary MSC code (in curly brackets) is mandatory,
% while secondary MSC codes (in square brackets) are optional.
%------
\classification[15B48, 15A18]{15A42}

%------
% Add a list of keywords.
%------

%------
% Insert your abstract.
%------
\begin{abstract}
    In the last decade, numerous works have investigated several properties of symplectic eigenvalues. 
    Remarkably, the results on symplectic eigenvalues have been found to be analogous to those of eigenvalues of Hermitian matrices with appropriate interpretations. 
    In particular, symplectic analogs of famous eigenvalue inequalities are known today such as Weyl's inequalities, Lidskii's inequalities, and Schur--Horn majorization inequalities. 
    In this paper, we provide necessary and sufficient conditions for equality in the symplectic analogs of the aforementioned inequalities. 
\end{abstract}

\maketitle

%------
% INSERT THE BODY OF THE PAPER HERE (except
% acknowledgments, funding info and bibliography)
%------

\section{Introduction}
    Let $\mathbb{M}_n(\mathbb{R})$ be the set of $n \times n$ real matrices, and $\mathbb{P}_n(\mathbb{R})$ denote the subset of $\mathbb{M}_n(\mathbb{R})$ consisting of symmetric positive definite matrices. 
    Denote by $\operatorname{Sp}(2n)$ the real symplectic group defined as
    \begin{align}
        \operatorname{Sp}(2n) \coloneqq 
            \left\{
                M \in \mathbb{M}_{2n}(\mathbb{R}): M^T J_{2n} M=J_{2n}
            \right\},
    \end{align}
    where  
    $J_{2n} \coloneqq  
    \left(
            \begin{smallmatrix}
                0 & 1 \\ 
                -1 & 0
            \end{smallmatrix} 
    \right) \otimes I_n$,
    $I_n$ being the $n \times n$ identity matrix. 
    Given $A \in \mathbb{P}_{2n}(\mathbb{R})$, there exists $M \in \operatorname{Sp}(2n)$ such that
    \begin{align}\label{williamson_decomposition}
        M^T A M = D \oplus D,
    \end{align}
    where $D$ is an $n\times n$ diagonal matrix with positive diagonal entries $d_1(A) \leq  \cdots \leq d_n(A)$, called the symplectic eigenvalues of $A$.
    This result is known as Williamson's theorem  \cite{williamson1936algebraic}. Several elementary proofs of Williamson's theorem are known today \cite{folland1989harmonic, simon1999congruences, ikramov2018symplectic}.

    Symplectic eigenvalues are ubiquitous in many areas of mathematics and physics such as classical Hamiltonian dynamics \cite{arnold1989graduate}, quantum mechanics \cite{dms}, and symplectic topology \cite{degosson, hofer}. 
    It also plays a crucial role in continuous-variable quantum information theory \cite{serafini2017quantum}, especially in Gaussian quantum information theory \cite{adesso2004extremal, chen2005gaussian, p, nicacio2021williamson, hsiang2022entanglement}. These applications have led to a growing interest in symplectic eigenvalues amongst mathematicians and physicists.
    In the last decade, numerous works have investigated properties of symplectic eigenvalues \cite{bhatia2015symplectic, HIAI2018129, mishra2020first, bhatia2020schur, bhatia_jain_2021, jain2021sums, jm, paradan2022, mishra2023}. 
    The notion of symplectic eigenvalues has also been extended to positive linear operators on infinite dimensional separable Hilbert spaces \cite{bhat2019real}, and some properties similar to the finite dimensional case are established \cite{john2022interlacing}. Remarkably, the results on symplectic eigenvalues are found to be analogous to those of eigenvalues of Hermitian matrices with appropriate interpretations. 
    In particular, symplectic analogs of famous eigenvalue inequalities are known today such as Weyl's inequalities \cite{bhatia_jain_2021}, Lidskii's inequalities \cite{jm, jain2021sums}, and Schur--Horn majorization inequalities \cite{bhatia2020schur}. 
 
    The goal of this paper is to provide necessary and sufficient conditions for equality in the symplectic analogs of the Weyl's, Lidskii's, and Schur--Horn majorization inequalities.  
    Interestingly, the conditions obtained for the symplectic Weyl's inequalities turn out to be analogous to its counterpart in the context of eigenvalues of Hermitian matrices obtained by Massey~et~al.~\cite{massey2014optimal}. 
    We also obtain a weaker symplectic analog of the necessary and sufficient conditions, obtained by Friedland~\cite{friedland2015equality}, for the equality in Lidskii's inequalities which serves as a set of necessary conditions for the equality to hold in symplectic Lidskii's inequalities.

    The organization of the paper is as follows. 
    We begin by reviewing a basic theory of symplectic eigenvalues and symplectic matrices in Section~\ref{sec:background}. 
    In Section~\ref{sec:symplectic_weyl_equality}, we derive precise conditions for equality in the symplectic analog of Weyl's inequalities $(\text{Theorem}~\ref{thm:symp_weyl_equalities})$. 
    In Section~\ref{sec:symplectic_lidskii_equalities}, we first provide necessary and sufficient conditions for equality in symplectic Lidskii's inequalities (Theorem~\ref{thm:equality-sym-lidskii-inequalities}). 
    We then establish some interesting properties of \textit{symplectic subspaces} associated to a given positive definite matrix (Propositions~\ref{prop:inv_sym_subspace_property},~\ref{prop:piecewise-analyticity-sym-subspace}). 
    Lastly, we establish a new set of necessary conditions for equality to hold in the symplectic Lidskii's inequalities (Theorem~\ref{thm:necessary-sym-lidskii}).
    In Section~\ref{sec:schur--horn}, we provide exact description of positive definite matrices saturating the weak supermajorization by majorization in a symplectic analog of the classic Schur--Horn theorem (Theorem~\ref{thm:sym_schur_horn}). 

%%%%%%%%%%%%%%%%%%%%%%%%%%%%%%%%%%%%%%%%%%%%%%%%%%%%%%%
 
\section{Background}\label{sec:background}
    We shall omit the subscript $2n$ in $J_{2n}$ whenever the size of the matrix is clear from the context.
    We call a pair of vectors $(u,v)$ \textit{ symplectically normalized} if it satisfies $\langle u, Jv\rangle=1$, where $\langle \cdot, \cdot \rangle$ denotes the Euclidean inner product. Two pairs of vectors $(u_1, v_1)$, $(u_2, v_2)$ are called \textit{ symplectically orthogonal} to each other if
    \begin{align}
        \langle u_i, J v_j \rangle = \langle u_i, J u_j \rangle = \langle v_i, J v_j \rangle = 0
    \end{align}
    for all $i \neq j$, $i,j=1,2$.
    A subset $\{u_1,\ldots, u_k, v_1,\ldots, v_k\}$ of $\mathbb{R}^{2n}$ is called \textit{ symplectically orthogonal} if the pairs of vectors $(u_1, v_1),\ldots, (u_k, v_k)$ are symplectically orthogonal to each other. 
    The subset is called \textit{ symplectically orthonormal} if it is symplectically orthogonal and the pairs of vectors $(u_1, v_1),\ldots, (u_k, v_k)$ are normalized. 
    A symplectically orthonormal subset of $\mathbb{R}^{2n}$ consisting of $2n$ vectors is called a \textit{ symplectic basis} of $\mathbb{R}^{2n}$. 
    There is a one-to-one correspondence between symplectic bases of $\mathbb{R}^{2n}$ and the symplectic group $\operatorname{Sp}(2n)$---symplectic bases $\{u_1,\ldots, u_n, v_1,\ldots, v_n \}$ of $\mathbb{R}^{2n}$ correspond to the matrices $[u_1,\ldots, u_n, v_1,\ldots, v_n]$ in $\operatorname{Sp}(2n)$. 
    The symplectic group is closed under matrix transpose and inverse, and every matrix in the symplectic group is called a symplectic matrix.
    See \cite{folland1989harmonic, degosson} for a review of symplectic linear algebra.
    We denote by $\operatorname{Sp}(2n, 2k)$ the set of $2n \times 2k$ real  matrices $M$ which satisfy $M^T J_{2n}M=J_{2k}$; in particular, $\operatorname{Sp}(2n,2n)=\operatorname{Sp}(2n)$. 
    Denote by $\operatorname{OrSp}(2n)$ the set of $2n \times 2n$ real \textit{ orthosymplectic} (orthogonal and symplectic) matrices.  
    Given any matrix $R \in \mathbb{M}_n(\mathbb{R})$, let $\operatorname{Ran}(R)$ denote the vector subspace of $\mathbb{R}^{n}$ spanned by the columns of $R$.

    Williamson's theorem can be equivalently stated as follows \cite[Proposition~2.1]{jm}: 
    for $A \in \mathbb{P}_{2n}(\mathbb{R})$ there exists a symplectic basis $\{u_1,\ldots, u_n, v_1, \ldots, v_n\}$ of $\mathbb{R}^{2n}$ satisfying
    \begin{align}\label{williamson_alternate_form}
        A u_i = d_i(A) J v_i, \quad Av_i = -d_i(A) J u_i, \qquad \text{for } i=1,\ldots, n.
    \end{align}
    We call a pair of non-zero vectors $(u_i, v_i)$ satisfying \eqref{williamson_alternate_form} a symplectic eigenvector pair of $A$ corresponding to the symplectic eigenvalue $d_i(A)$. 
    In addition, if the pair $(u_i, v_i)$ is normalized, it is called a normalized symplectic eigenvector pair of $A$ corresponding to $d_i(A)$. 
    We call a symplectic basis $\{u_1,\ldots, u_n, v_1,\ldots, v_n\}$ of $\mathbb{R}^{2n}$ satisfying \eqref{williamson_alternate_form} a \textit{ symplectic eigenbasis} of $A$. 

    A subspace $\mathscr{U}$ of $\mathbb{R}^{2n}$ is called a symplectic subspace of $\mathbb{R}^{2n}$ if for any $u \in \mathscr{U}$ there exists $v \in \mathscr{U}$ such that $\langle u, Jv \rangle \neq 0$. 
    Every symplectic subspace $\mathscr{U}$ of $\mathbb{R}^{2n}$ has a symplectically orthonormal basis, which we simply call a symplectic basis of $\mathscr{U}$. 
    Also, the span of any symplectically orthonormal subset of $\mathbb{R}^{2n}$ is a symplectic subspace of $\mathbb{R}^{2n}$. Consequently, every symplectic subspace of $\mathbb{R}^{2n}$ is of even dimension.
    See \cite[Section~1.2.1]{degosson}.

 %%%%%%%%%%%%%%%%%%%%%%%%%%%%%%%%%%%%%%%%%%%%%%%%%%%%%%%

\section{Equality in symplectic Weyl's inequalities}\label{sec:symplectic_weyl_equality}
    Let $\mathbb{H}_n(\mathbb{C})$ denote the set of $n \times n$ complex Hermitian matrices. 
    For $X \in \mathbb{H}_n(\mathbb{C})$, let $\lambda_1(X) \geq \cdots \geq \lambda_n(X)$ denote the eigenvalues of $X$ in the descending order.
    The famous Weyl's  eigenvalue inequalities \cite{weyl_inequalities} for Hermitian matrices state that 
    for any $X,Y \in \mathbb{H}_n(\mathbb{C})$ and indices $i,j \in \{1,\ldots, n\}$ such that $i+j-1 \leq n$, we have
    \begin{align}\label{weyl_inequalities}
	\lambda_{i+j-1}(X+Y) \leq \lambda_i(X) + \lambda_{j}(Y).
    \end{align}
    See Section~III.2 of \cite{ma_bhatia}.
    In the work of Massey~et al.~\cite{massey2014optimal}, it was shown that the equality in \eqref{weyl_inequalities} holds if and only if there is a common eigenvector of $X$, $Y$, and $X+Y$ corresponding to their eigenvalues $\lambda_i(X)$, $\lambda_{j}(Y)$, and $\lambda_{i+j-1}(X+Y)$.

    A symplectic analog of Weyl's inequalities was recently given by Bhatia and Jain~\cite{bhatia_jain_2021}, which states that for $A, B \in \mathbb{P}_{2n}(\mathbb{R})$ and indices $i,j \in \{1, \ldots, n\}$ satisfying $i+j-1 \leq n$,  we have
    \begin{align}\label{symp_weyl_inequalities}
	d_{i+j-1}(A+B) \geq d_i(A) + d_j(B).
    \end{align}
    In this section, we derive a necessary and sufficient condition for the equality to hold in \eqref{symp_weyl_inequalities}. 
    Interestingly, the condition obtained is analogous to that of the eigenvalues.

    The following observation will be useful in proving the main result of the section.
%----------------------------------------------------------------------------------------------------------------------%
\begin{proposition}\label{prop:span_symp_eigenvectors}
    Let $A \in \mathbb{P}_{2n}(\mathbb{R})$ and let $d>0$ be a symplectic eigenvalue of $A$. 
    Suppose $(u_1,v_1),\ldots, (u_k, v_k)$ are normalized symplectic eigenvector pairs of $A$ corresponding to the common symplectic eigenvalue $d$. 
    Any non-zero pair of vectors $(u,v)$ of the form $u = \sum_{\ell=1}^k (\alpha_\ell u_\ell+\beta_\ell v_\ell)$ and $v=\sum_{\ell=1}^k (-\beta_\ell u_\ell+\alpha_\ell v_\ell)$ is also a symplectic eigenvector pair of $A$ corresponding to $d$.
\end{proposition}
%----------------------------------------------------------------------------------------------------------------------%
\begin{proof}
    Let $(u,v)$ be a pair of non-zero vectors in the given form. 
    Using the equivalent form of Williamson's theorem \eqref{williamson_alternate_form}, we get
    \begin{align}
        Au 	&= \sum_{\ell=1}^{k} (\alpha_\ell Au_\ell+\beta_\ell Av_\ell) \\
        		&= d J \sum_{\ell=1}^{k} (-\beta_\ell u_\ell+ \alpha_\ell v_\ell)  \\
        		&= d J v.
    \end{align}
    Similarly, we also get $Av=-dJu$.
\end{proof}
%----------------------------------------------------------------------------------------------------------------------%

%----------------------------------------------------------------------------------------------------------------------%
\begin{theorem}\label{thm:symp_weyl_equalities}
    Let $A, B \in \mathbb{P}_{2n}(\mathbb{R})$ and $i,j \in \{1,\ldots, n\}$ be any indices such that $i+j-1 \leq n$. 
    Then the equality holds in \eqref{symp_weyl_inequalities}, i.e.,
    \begin{equation}\label{symp_weyl_equalities}
        d_{i+j-1}(A+B) = d_i(A) + d_j(B),
    \end{equation}
    if and only if there exists a common normalized symplectic eigenvector pair of $A$, $B$, and $A+B$ corresponding to their symplectic eigenvalues $d_i(A)$, $d_j(B)$, and $d_{i+j-1}(A+B)$. 
\end{theorem}
%----------------------------------------------------------------------------------------------------------------------%
\begin{proof}
    The ``if'' direction is straightforward. Indeed, assume that there exists a common normalized symplectic eigenvector pair $(u,v)$ of $A$, $B$, and $A+B$ corresponding to their symplectic eigenvalues $d_i(A)$, $d_j(B)$, and $d_{i+j-1}(A+B)$. 
    From the equivalent formulation of Williamson's theorem \eqref{williamson_alternate_form}, we get
    \begin{align}
    d_{i+j-1}(A+B) 
    	&= \dfrac{\langle u, (A+B) u \rangle + \langle v, (A+B) v \rangle}{2}   \\
    	&= \dfrac{\langle u, A u \rangle + \langle v, A v \rangle}{2} + \dfrac{\langle u, B u \rangle + \langle v, B v \rangle}{2}   \\
    	&= d_{i}(A) + d_j(B).
    \end{align}

    The ``only if" direction is more interesting.
    \sloppy Fix symplectic eigenbases $\{u_1,\ldots, u_n, v_1,\ldots, v_n\}$, $\{w_1,\ldots, w_n, x_1,\ldots, x_n\}$, and $\{y_1,\ldots, y_n, z_1,\ldots, z_n\}$ of $A$, $B$, and $A+B$, respectively. 
    From the proof of Theorem~1.2 of \cite{bhatia_jain_2021}, we get a pair of non-zero vectors $(x, x')$ of the form
    \begin{align}
    x &= \sum_{\ell=1}^{i+j-1} (\alpha_\ell y_\ell + \beta_\ell z_\ell), \label{eq:vector_x} \\ 
    x' &= \sum_{\ell=1}^{i+j-1} (-\beta_\ell y_\ell + \alpha_\ell z_\ell),\label{eq:vector_x'}
    \end{align}
    such that
    \begin{align}
    \langle x, J x' \rangle = \sum_{\ell=1}^{i+j-1} \left(\alpha_\ell^2 + \beta_\ell^2 \right) =1,\label{eq:symplectic_normalization}
    \end{align}
    and $x$ lies in the intersection of the vector subspaces spanned by 
    $\{u_i,\ldots, u_n, v_i,\ldots, v_n\}$, $\{w_j,\ldots, w_n, x_j,\ldots, x_n\},$ and $\{y_1,\ldots, y_{i+j-1}, z_1,\ldots, z_{i+j-1}\}$.
    By Propositions~3.2 and 3.3 of \cite{bhatia_jain_2021}, we have
    \begin{align}
    d_{i+j-1}(A+B) 
        	&\geq \dfrac{\langle x, (A+B)x \rangle + \langle x^{\prime}, (A+B)x^{\prime} \rangle}{2},\label{eqn13} \\
    		\dfrac{\langle x, Ax \rangle + \langle x^{\prime}, Ax^{\prime} \rangle}{2} 
        	&\geq d_{i}(A), \label{eqn6}\\  
    		\dfrac{\langle x, Bx \rangle + \langle x^{\prime}, Bx^{\prime} \rangle}{2}
        	&\geq d_{j}(B).\label{eqn7}  
    \end{align}

    Assume that the equality \eqref{symp_weyl_equalities} holds. 
    This implies that each of the inequalities \eqref{eqn13}, \eqref{eqn6}, and \eqref{eqn7} is equality; i.e., we have
    \begin{align}
    d_{i+j-1}(A+B) 	&= \dfrac{\langle x, (A+B)x \rangle + \langle x^{\prime}, (A+B)x^{\prime} \rangle}{2}, \label{equality:a+b} \\
    d_{i}(A) 			&= \dfrac{\langle x, Ax \rangle + \langle x^{\prime}, Ax^{\prime} \rangle}{2}, \label{equality:a} \\
    d_{j}(B) 			&= \dfrac{\langle x, Bx \rangle + \langle x^{\prime}, Bx^{\prime} \rangle}{2}. \label{equality:b}
    \end{align}
    From the equivalent formulation of Williamson's theorem \eqref{williamson_alternate_form}, and the descriptions \eqref{eq:vector_x} and \eqref{eq:vector_x'} of $(x,x')$, we get
    \begin{align}
    	\dfrac{\langle x, (A+B)x \rangle + \langle x^{\prime}, (A+B)x^{\prime} \rangle}{2} 
		= \sum_{\ell=1}^{i+j-1} d_\ell(A+B) \left(\alpha_\ell^2+\beta_\ell^2\right).\label{eq:equality a+b-aux}
    \end{align}
    From \eqref{equality:a+b} and \eqref{eq:equality a+b-aux}, we thus get
    \begin{align}
    	\sum_{\ell=1}^{i+j-1} d_\ell(A+B) \left(\alpha_\ell^2+\beta_\ell^2\right) &=d_{i+j-1}(A+B).
    \end{align}
    Using the condition \eqref{eq:symplectic_normalization} in the above equation, we get $\alpha_\ell=0$ and $\beta_\ell=0$ whenever $d_\ell(A+B)< d_{i+j-1}(A+B)$. 
    So, Proposition~\ref{prop:span_symp_eigenvectors} implies that $(x, x')$ is a normalized symplectic eigenvector pair of $A+B$ corresponding to the symplectic eigenvalue $d_{i+j-1}(A+B)$.

    In the remainder of the proof, we argue that $(x,x')$ is also a normalized symplectic eigenvector pair of $A$ and $B$ corresponding to their symplectic eigenvalues $d_i(A)$ and $d_j(B)$.
    We can express the vectors $x$ and $x^\prime$ in the symplectic eigenbasis of $A$ as
    \begin{align}
    x &= \sum_{\ell=i}^{n} \left(\gamma_\ell u_\ell + \delta_\ell v_\ell \right), \\ 
    x' &= \sum_{\ell=1}^{n} (\gamma_\ell' u_\ell + \delta_\ell' v_\ell),
    \end{align}
    where we used the fact that $x$ lies in the span of $\{u_i,\ldots, u_n, v_i, \ldots, v_n\}$.
    This gives
    \begin{align}
    	&\dfrac{\langle x, Ax \rangle + \langle x^{\prime}, Ax^{\prime} \rangle}{2} \nonumber \\
        	&\hspace{0.5cm}= \dfrac{1}{2}\sum_{\ell=i}^{n} d_\ell(A) \left(\gamma^2_{\ell}+\delta^2_{\ell}+(\gamma^{\prime}			_{\ell})^2+(\delta^{\prime}_{\ell})^2\right)  +  \dfrac{1}{2}\sum_{\ell=1}^{i-1} d_\ell(A)\left((\gamma^{\prime}			_{\ell})^2+(\delta^{\prime}_{\ell})^2\right) \label{eq:2di-a-1}  \\
        	&\hspace{0.5cm} \geq  \dfrac{d_i(A)}{2}\sum_{\ell=i}^{n} \left(\gamma^2_{\ell}+\delta^2_{\ell}+(\gamma^{\prime}			_{\ell})^2+(\delta^{\prime}_{\ell})^2\right) + \dfrac{1}{2}\sum_{\ell=1}^{i-1} d_\ell(A)\left((\gamma^{\prime}				_{\ell})^2+(\delta^{\prime}_{\ell})^2\right)  \label{eqn16}  \\
        	&\hspace{0.5cm} \geq d_i(A) \sum_{\ell=i}^{n} \left(\gamma_\ell \delta^{\prime}_\ell - \delta_\ell \gamma^{\prime}			_\ell\right) + \dfrac{1}{2}\sum_{\ell=1}^{i-1} d_\ell(A)\left((\gamma^{\prime}_{\ell})^2+(\delta^{\prime}				_{\ell})^2\right)  \label{eqn17}\\
        	&\hspace{0.5cm} = d_i(A) \big\langle x, Jx' \big\rangle + \dfrac{1}{2} \sum_{\ell=1}^{i-1} d_\ell(A)							\left((\gamma^{\prime}_{\ell})^2+(\delta^{\prime}_{\ell})^2\right)  \\
        	&\hspace{0.5cm} = d_i(A) + \dfrac{1}{2}\sum_{\ell=1}^{i-1} d_\ell(A)\left((\gamma^{\prime}_{\ell})^2+(\delta^{\prime}			_{\ell})^2\right) \label{eq:2di-a-2}.
    \end{align}
    The equality \eqref{equality:a} implies that each inequality in the development \eqref{eq:2di-a-1}--\eqref{eq:2di-a-2} is equality. 
    So, we must have $\gamma_\ell'=\delta_\ell'=0$ for $\ell=1,\ldots, i-1$. 
    Furthermore, \eqref{eqn16} being equality implies $\gamma_\ell=\delta_\ell=\gamma_\ell'=\delta_\ell'=0$ whenever $d_\ell(A)> d_i(A)$. 
    Also, \eqref{eqn17} being equality implies that $\gamma_\ell=\delta_\ell'$ and $\delta_\ell=-\gamma_\ell'$.
    It then follows from Proposition~\ref{prop:span_symp_eigenvectors} that $(x, x')$ is a normalized symplectic eigenvector pair of $A$ corresponding to the symplectic eigenvalue $d_i(A)$.
    A similar argument shows that $(x,x')$ is a normalized symplectic eigenvector pair of $B$ corresponding to the symplectic eigenvalue $d_j(B)$.
\end{proof}
%----------------------------------------------------------------------------------------------------------------------%

%%%%%%%%%%%%%%%%%%%%%%%%%%%%%%%%%%%%%%%%%%%%%%%%%%%%%%%

\section{Equality in symplectic Lidskii's inequalities}\label{sec:symplectic_lidskii_equalities}
    The well-known Lidskii's inequalities \cite{lidskii1950characteristic, wielandt1955} for eigenvalues are given as follows: 
    for any $X,Y \in \mathbb{H}_n(\mathbb{C})$ and distinct indices $1 \leq i_1 < \cdots < i_k \leq n$, we have
    \begin{align}\label{eq:lidskii-inequality}
    \sum_{j=1}^{k} \lambda_{i_j}(X+Y) \leq \sum_{j=1}^{k} \lambda_{i_j}(X)+\sum_{j=1}^{k} \lambda_{j}(Y).
    \end{align}
    A symplectic analog of the Lidskii's inequalities was given by Jain and Mishra \cite{jm}, which states that,
    for any $A, B \in \mathbb{P}_{2n}(\mathbb{R})$ and $k$ distinct indices $1 \leq i_1 < \cdots < i_k \leq n$, we have
    \begin{align}\label{eq:sym-lidskii-inequality}
        \sum_{j=1}^{k} d_{i_j}(A+B) \geq \sum_{j=1}^{k} d_{i_j}(A)+\sum_{j=1}^{k} d_{j}(B).
    \end{align}
    The following theorem states necessary and sufficient conditions for the equality to hold in \eqref{eq:sym-lidskii-inequality}. 
    Part of the proof of the theorem is inspired by ideas presented in the proof of Theorem~3.1 of \cite{friedland2015equality}.
%----------------------------------------------------------------------------------------------------------------------%
\begin{theorem}\label{thm:equality-sym-lidskii-inequalities}
    Let $A, B \in \mathbb{P}_{2n}(\mathbb{R})$ and $1 \leq i_1 < \cdots < i_k \leq n$ be any distinct indices. 
    The following statements are equivalent:
    \begin{itemize}
	\item [(i)] 
		The equality in \eqref{eq:sym-lidskii-inequality} holds, that is,
		\begin{align}\label{eq:sym-lidskii-equality}
    			\sum_{j=1}^{k} d_{i_j}(A+B) = \sum_{j=1}^{k} d_{i_j}(A)+\sum_{j=1}^{k} d_{j}(B).
		\end{align}
	\item [(ii)] 
		For all $t \in [0,1]$, we have
		\begin{align}\label{eq:sym:lidskii-linear}
          		\sum_{j=1}^{k} d_{i_j}(A+tB) = \sum_{j=1}^{k} d_{i_j}(A)+t\sum_{j=1}^{k} d_{j}(B).
		\end{align} 
	\item [(iii)] 
		There exist positive numbers $0=a_0<a_1<\cdots < a_r=1$ such that the following properties hold:
		for all $\ell \in \{1,\ldots, r\}$ and $t \in (a_{\ell-1}, a_\ell)$, there exists $M(t) \in \operatorname{Sp}(2n, 2k)$ 	such that the columns of $M(t)$ form symplectic eigenvector pairs of $A+t B$ corresponding to the symplectic eigenvalues $d_{i_1}(A+t B), \ldots, d_{i_k}(A+t B)$, and it satisfies 
		\begin{align}
           		\operatorname{Tr}\! \left[M(t)^T A M(t) \right]
				&= 2\sum_{j=1}^{k} d_{i_j}(A), \label{eq1:conditions-equality-sym-lidskii} \\
          		\operatorname{Tr}\!\left[M(t)^T B M(t) \right]
				&=2\sum_{j=1}^{k} d_{j}(B). \label{eq2:conditions-equality-sym-lidskii}
		\end{align} 
    \end{itemize}
\end{theorem}
%----------------------------------------------------------------------------------------------------------------------%
\begin{proof}
    By Theorem~4.7 of \cite{jm}, the symplectic eigenvalue maps $d_i: [0,1] \to \mathbb{R}$ given by $d_i(t) \coloneqq d_i(A+tB)$  for $i=1,\ldots,n$ and $t \in [0,1]$ are piecewise analytic on $[0,1]$. 
    Furthermore, there exist piecewise analytic functions $u_i, v_i:[0,1] \to \mathbb{R}^{2n}$ for $i=1,\ldots,n$ such that $\{u_1(t),\ldots, u_n(t),v_1(t),\ldots, v_n(t)\}$ is a symplectic eigenbasis of $A+tB$ for $t \in [0,1]$. 
    Therefore, there exist finitely many positive numbers $0=a_0<a_1<\cdots < a_r=1$ such that $d_i,$ $u_i$, and $v_i$ are analytic on each open interval $(a_{\ell-1}, a_{\ell})$ for $i=1,\ldots,n$ and $\ell=1,\ldots, r$.
    Define $\phi:[0,1] \to \mathbb{R}$  as
    \begin{align}
        \phi(t)\coloneqq \sum_{j=1}^k d_{i_j}(A+tB).\label{eq:phi_t_def}
    \end{align}
    By differentiating \eqref{eq:phi_t_def} with respect to $t$ on $[0,1]\backslash \{a_0, a_1, \ldots, a_r\}$, and using Eq.~(5.8) of \cite{jm} for the derivative expressions of the symplectic eigenvalue maps, we get 
    \begin{align}
        \phi^\prime(t) = \dfrac{1}{2}  \operatorname{Tr}\! \left[ M (t)^T B M(t) \right],\label{eq:phi_derivative}
    \end{align}
    where $M(t) \coloneqq \left[ u_{i_1}(t), \ldots, u_{i_k}(t), v_{i_1}(t),\ldots,v_{i_k}(t) \right] \in \operatorname{Sp}(2n, 2k)$.
    We know from Theorem~5 of \cite{bhatia2015symplectic} that
    \begin{align}
        \dfrac{1}{2}  \operatorname{Tr}\! \left[ M (t)^T B M(t) \right] 
        		\geq \sum_{i=1}^k d_i(B).\label{eq:extremal_char_inequality}
    \end{align}
    Combining \eqref{eq:phi_derivative} and \eqref{eq:extremal_char_inequality} we thus get
    \begin{align}
        \phi^\prime(t) \geq \sum_{i=1}^k d_i(B), \qquad t \in [0,1]\backslash \{a_0, a_1, \ldots, a_r\}.\label{eq:phi-prime-inequality}
    \end{align}
    
    It trivially follows that $(ii) \Rightarrow (i)$. The following arguments show that $(i) \Rightarrow (ii)$. 
    Assume that the equality \eqref{eq:sym-lidskii-equality} holds. 
    By using the definition of $\phi$ in \eqref{eq:phi-prime-inequality}, we then get
    \begin{align}
        \phi^\prime(t) \geq \phi(1)-\phi(0), \qquad t \in [0,1]\backslash \{a_0, a_1, \ldots, a_r\}.\label{eq:phi_inequality_0_1}
    \end{align}
    The function $\phi$ is continuous on $[0,1]$, which follows from continuity of symplectic eigenvalues \cite{jm}. 
    From the fundamental theorem of calculus, we get
    \begin{align}\label{eq:phi-prime-integration-zero}
        \int_{0}^{1}  \phi^{\prime}(t)  \operatorname{d}\!t=\phi(1)-\phi(0).
    \end{align}
    Combining \eqref{eq:phi_inequality_0_1} and \eqref{eq:phi-prime-integration-zero}, we get
    \begin{align}
        \phi^{\prime}(t)=\phi(1)-\phi(0), \qquad t \in [0,1]\backslash \{a_0,a_1,\ldots, a_r\}. \label{eq:phi-prime-equality}
    \end{align}
    Since $\phi$ is continuous, it thus follows from \eqref{eq:phi-prime-equality} that
    \begin{align}\label{eq:phi-form}
        \phi(t)=\phi(0)+t (\phi(1)-\phi(0)), \qquad t \in [0,1],
    \end{align}
    implying \eqref{eq:sym:lidskii-linear} holds.
    This proves $(ii) \Rightarrow (i)$.
    We have thus established the equivalence $(i) \Leftrightarrow (ii)$.
    
    We now show that $(ii) \Rightarrow (iii)$. 
    Assume that $\eqref{eq:sym:lidskii-linear}$ holds.
    Let $t \in (a_{\ell-1}, a_{\ell})$ be arbitrary. 
    By construction, the columns of $M(t)$ form symplectic eigenvector pairs of $A+tB$ corresponding to its symplectic eigenvalues $d_{i_1}(A+t B),\ldots, d_{i_k}(A+t B)$.  
    We now argue that $M(t)$ satisfies the conditions \eqref{eq1:conditions-equality-sym-lidskii} and \eqref{eq2:conditions-equality-sym-lidskii}. 
    We have
    \begin{align}
        \operatorname{Tr}\left[M(t)^T B M(t) \right]  
        		&= 2\phi^{\prime}(t)   \\
        		&= 2\left(\phi(1)-\phi(0)\right) \\
        		&= 2\left(\sum_{j=1}^k d_{i_j}(A+B)-\sum_{j=1}^k d_{i_j}(A)\right)  \\
        		&= 2\sum_{j=1}^k d_j(B) \label{eq_v1:first-equality-condition-5}.
    \end{align}
    The first equality follows from \eqref{eq:phi_derivative}, the second equality follows from \eqref{eq:phi-prime-equality}, the third equality follows from the definition \eqref{eq:phi_t_def}, and the last equality follows by substituting $t=1$ into \eqref{eq:sym:lidskii-linear}.
    Also, we have
    \begin{align}\label{eq_v1:second-equality-condition-1}
        \operatorname{Tr}\left[M(t)^TAM(t) \right]
        		&=  \operatorname{Tr}\left[M(t)^T(A+tB) M(t) \right] - t \operatorname{Tr}\left[M(t)^TB M(t) \right].
    \end{align}
    Since the columns of $M(t)$ form symplectic eigenvector pairs of $A+tB$ corresponding to its symplectic eigenvalues $d_{i_1}(A+t B), \ldots, d_{i_k}(A+t B)$, the first term in the right-hand side of \eqref{eq_v1:second-equality-condition-1} is given by
    \begin{align}\label{eq_v1:second-equality-aux}
         \operatorname{Tr}\left[M(t)^T(A+t B)M(t) \right]= 2\sum_{j=1}^k d_{i_j}(A+t B).
    \end{align}
    Substituting \eqref{eq_v1:second-equality-aux} into \eqref{eq_v1:second-equality-condition-1} and using the relation \eqref{eq_v1:first-equality-condition-5}, we get
    \begin{align}
        \operatorname{Tr}\left[M(t)^TAM(t) \right]
        		&=  2\sum_{j=1}^k d_{i_j}(A+t B) -  2t\sum_{j=1}^k d_j(B).
    \end{align}
    Using the assumption \eqref{eq:sym:lidskii-linear}, we then get
    \begin{align}
        \operatorname{Tr}\left[M(t)^TAM(t) \right]
        		&=  2\sum_{j=1}^k d_{i_j}(A).
    \end{align}
    The other direction $(iii) \Rightarrow (ii)$ is rather straightforward. We have thus established $(ii) \Leftrightarrow (iii)$.
\end{proof}
%----------------------------------------------------------------------------------------------------------------------%

    It was shown by Friedland~\cite{friedland2015equality} that the equality holds in the Lidskii's inequality \eqref{eq:lidskii-inequality} if and only if there exist $k$-dimensional subspaces $\mathscr{X}_1 ,\ldots, \mathscr{X}_r \subset \mathbb{C}^n$ and positive numbers $0=b_0< b_1<\cdots < b_r=1$ satisfying the following conditions:
    \begin{itemize}
        \item [(i)] The subspaces $\mathscr{X}_1 ,\ldots, \mathscr{X}_r$ are invariant under both $X$ and $Y$.
        \item [(ii)] Each subspace $\mathscr{X}_\ell$ for $\ell \in \{1,\ldots, r\}$ is spanned by $k$ orthonormal eigenvectors of $Y$ corresponding to its eigenvalues $\lambda_1(Y), \ldots , \lambda_k(Y)$.
        \item [(iii)] For each $\ell \in \{1,\ldots, r\}$ and $t\in [b_{\ell-1}, b_{\ell}]$, the eigenvalues of the restriction of $X+tY$ to $\mathscr{X}_\ell$ are $\lambda_{i_1}(X+tY), \ldots, \lambda_{i_k}(X+tY)$.
    \end{itemize}
    In the remainder of this section, we argue that a symplectic analog of the aforementioned conditions (with a weaker analog of condition $(iii)$) serves as a necessary condition for equality to hold in symplectic Lidskii's inequality \eqref{eq:sym-lidskii-inequality}.
    It is still open to determine if the given necessary condition for the equality to hold in symplectic Lidskii's inequality \eqref{eq:sym-lidskii-inequality} is sufficient as well.

    We begin by establishing some properties of symplectic eigenvalues related to symplectic subspaces of $\mathbb{R}^{2n}$ which are of independent interest. 
    It is well-known that every invariant subspace of a Hermitian matrix has an orthonormal basis of eigenvectors of the Hermitian matrix \cite{horn2012matrix}. The following is a symplectic analog of this fact. 
%----------------------------------------------------------------------------------------------------------------------%
\begin{proposition}\label{prop:inv_sym_subspace_property}
    \sloppy    
    Let $A \in \mathbb{P}_{2n}(\mathbb{R})$ and $\mathscr{U}$ be a $2k$-dimensional symplectic subspace of $\mathbb{R}^{2n}$. 
    If $\mathscr{U}$ is an invariant subspace of $JA$ then there exists a symplectic basis $\{u_1,\ldots, u_k, v_1,\ldots, v_k\}$ of $\mathscr{U}$ such that $(u_1, v_1),\ldots, (u_k, v_k)$ are symplectic eigenvector pairs of $A$ corresponding to some symplectic eigenvalues $\gamma_1 \leq \cdots \leq \gamma_k$ of $A$. 
\end{proposition}
%----------------------------------------------------------------------------------------------------------------------%
%----------------------------------------------------------------------------------------------------------------------%
\begin{proof}
    Let $A \in \mathbb{P}_{2n}(\mathbb{R})$ and $\mathscr{U}$ be a $2k$-dimensional symplectic subspace of $\mathbb{R}^{2n}$.
    Assume that $\mathscr{U}$ is an invariant subspace of $JA$. 
    This implies that
\begin{align}
	A^{1/2} \left(\mathscr{U}+\iota \mathscr{U} \right) \coloneqq \{A^{1/2}u + \iota A^{1/2} v: u,v \in \mathscr{U}\} 				\subset \mathbb{C}^{2n}
\end{align}
    is an invariant subspace of the Hermitian matrix $\iota A^{1/2}JA^{1/2}$.
    \sloppy	
    Therefore, $A^{1/2} \left(\mathscr{U}+\iota \mathscr{U} \right)$ has an orthonormal basis of eigenvectors of $\iota A^{1/2}JA^{1/2}$.  
    
    It is known that the eigenvalues of $\iota A^{1/2}JA^{1/2}$ are $\pm d_1(A),\ldots, \pm d_n(A)$. 
    Moreover, for $x,y \in \mathbb{R}^{2n}$, $x + \iota y$  is an eigenvector of $\iota A^{1/2}JA^{1/2}$ corresponding to its eigenvalue $\lambda$ if and only if $x - \iota y$  is an eigenvector of $\iota A^{1/2}JA^{1/2}$ corresponding to its eigenvalue $-\lambda$ \cite[Lemma~2.2]{jm}.
    So, the eigenvalues of $\iota A^{1/2}JA^{1/2}$ restricted to the subspace $A^{1/2} \left(\mathscr{U}+\iota \mathscr{U} \right)$ occur in negative-positive pairs.
    Also, the dimension of $A^{1/2} \left(\mathscr{U}+\iota \mathscr{U} \right)$ as a subspace of $\mathbb{C}^{2n}$ is $2k$.
    Therefore, there exist $x_i, y_i \in \mathscr{U}$ for $i=1,\ldots,k$ such that $A^{1/2}x_1-\iota A^{1/2} y_1, \ldots, A^{1/2} x_k-\iota A^{1/2} y_k$ are orthonormal eigenvectors of $\iota A^{1/2}JA^{1/2}$ corresponding to some positive eigenvalues $\gamma_1 \leq \cdots \leq \gamma_k$.
    Proposition~2.3 of \cite{jm} then implies that the set $\{x_1, \ldots, x_k, y_1,\ldots, y_k\} \subset \mathscr{U}$ is symplectically orthogonal, and $(x_i, y_i)$ is a symplectic eigenvector pair of $A$ corresponding to its symplectic eigenvalue $\gamma_i$ satisfying
    \begin{align}
        \langle x_i, J y_i\rangle = \dfrac{1}{2\gamma_i}, \qquad i=1,\ldots, k.
    \end{align}
    Thus, choosing for each $i \in \{1,\ldots, k\}$,
    \begin{align}
        u_i &=\sqrt{2\gamma_i}x_i, \\
        v_i &= \sqrt{2\gamma_i}y_i,
    \end{align}
    gives the desired symplectic basis $\{u_1,\ldots, u_k, v_1,\ldots, v_k\}$ of $\mathscr{U}$.
\end{proof}
%----------------------------------------------------------------------------------------------------------------------%

%----------------------------------------------------------------------------------------------------------------------%
\begin{definition}
    Let $A$, $\mathscr{U}$, and $\gamma_1,\ldots, \gamma_k$ be as in Proposition~\ref{prop:inv_sym_subspace_property}.
    We call $\gamma_1,\ldots, \gamma_k$ the symplectic eigenvalues of $A$ associated with $\mathscr{U}$.
\end{definition}
%----------------------------------------------------------------------------------------------------------------------%
We know that for given $A, B \in \mathbb{P}_{2n}(\mathbb{R})$, the symplectic eigenvalue maps $[0,1] \ni t \mapsto d_i(A+tB)$ for $i=1,\ldots, n$ are picewise analytic \cite[Theorem~4.7]{jm}.
More generally, the symplectic eigenvalues of $A+tB$ associated with a fixed symplectic subspace are also piecewise analytic functions of $t$ in $[0,1]$, as stated in the following proposition.
%----------------------------------------------------------------------------------------------------------------------%
\begin{proposition}\label{prop:piecewise-analyticity-sym-subspace}
    Let $A, B \in \mathbb{P}_{2n}(\mathbb{R})$ and $\mathscr{U}$ be a $2k$-dimensional symplectic subspace of $\mathbb{R}^{2n}$ that is invariant under $A+tB$ for all $t \in [0,1]$. 
    Let $\gamma_{1}(t)\leq \cdots \leq \gamma_{k}(t)$ be the symplectic eigenvalues of $A+tB$ associated with $\mathscr{U}$ for $t \in  [0,1]$. 
    Then the maps $\gamma_{i}:[0,1] \to \mathbb{R}$ for $i=1,\ldots,k$ are piecewise analytic.
\end{proposition}
%----------------------------------------------------------------------------------------------------------------------%
\begin{proof}
    The eigenvalues of $\iota J(A+tB)$ restricted to $\mathscr{U}+\iota \mathscr{U}$ can be arranged such that they are piecewise analytic functions of $t$. 
    That is, there exist piecewise analytic functions $\mu_1,\ldots, \mu_{2k}$ on $[0,1]$ such that $\mu_1(t),\ldots, \mu_{2k}(t)$ are the eigenvalues of $\iota J(A+tB)$ restricted to $\mathscr{U}+\iota \mathscr{U}$ for all $t \in [0,1]$.
    See \cite[Chapter~2, Section~1]{kato2013perturbation}.
    We also know from Theorem~4.7 of \cite{jm} that the symplectic eigenvalue maps, defined by $d_i(t)\coloneqq d_i(A+tB)$ for $i=1,\ldots, n$ and $t \in [0,1]$ are piecewise analytic on $[0,1]$.
    So, there exist positive numbers $0=a_0 < a_1 < \cdots < a_r=1$ such that $\mu_1,\ldots, \mu_{2k}$ and $d_1,\ldots, d_n$ are analytic on each $(a_{\ell-1}, a_\ell)$ for $\ell=1,\ldots, r$.
    
    Let us fix an arbitrary $\ell \in \{1,\ldots, r\}$.
    Since the eigenvalues of $\iota J(A+tB)$ occur in negative-positive pairs and $\{\mu_1(t),\ldots, \mu_{2k}(t)\} \subset \{d_1(A+tB),\ldots, d_n(A+tB),-d_1(A+tB),\ldots, -d_n(A+tB)\}$ for all $t \in (a_{\ell-1}, a_\ell)$ \cite[Lemma~2.2]{jm}, it follows from the elementary theory of analytic functions \cite[Chapter~1]{krantz2002primer} that there exist distinct indices $\rho_1 \ldots, \rho_k \in \{1,\ldots, 2k\}$ and $1\leq \sigma_1 < \cdots < \sigma_k \leq n$ such that $\mu_{\rho_i}(t)=d_{\sigma_i}(A+tB)$ for all $t \in (a_{\ell-1}, a_\ell)$.
    Since the positive eigenvalues of $\iota J(A+tB)$ restricted to $\mathscr{U}+\iota \mathscr{U}$ are $ \gamma_{1}(t), \ldots, \gamma_{k}(t)$, we get $\gamma_i(t)=\mu_{\rho_i}(t)=d_{\sigma_i}(A+tB)$ for all $i=1,\ldots, k$ and $t \in (a_{\ell-1}, a_\ell)$. 
    Therefore, $\gamma_1,\ldots, \gamma_k$ are piecewise analytic on $[0,1]$.
\end{proof}
%----------------------------------------------------------------------------------------------------------------------%
\begin{remark}
    It is not clear if by assuming the given hypotheses of Proposition~\ref{prop:piecewise-analyticity-sym-subspace}, the symplectic eigenvalues of $A+tB$ associated with $\mathscr{U}$ can be arranged analytically as functions of $t \in [0,1]$. 
    This is because the matrix $\iota J(A+tB)$ is not normal, and the theory of analytic perturbation of non-normal matrices only gives piecewise analyticity of eigenvalues. 
    See Chapter~2 of \cite{kato2013perturbation} for a comprehensive study on the matter.
\end{remark}

%----------------------------------------------------------------------------------------------------------------------%
    We close this section by providing an interesting necessary condition for the equality to hold in \eqref{eq:sym-lidskii-inequality}.
%----------------------------------------------------------------------------------------------------------------------%
\begin{theorem}\label{thm:necessary-sym-lidskii}
    Let $A, B \in \mathbb{P}_{2n}(\mathbb{R})$ and $1 \leq i_1 < \cdots < i_k \leq n$ be any distinct indices such that equality holds in the symplectic Lidskii's inequality \eqref{eq:sym-lidskii-inequality}.
    Then there exist $2k$-dimensional symplectic subspaces $\mathscr{U}_1,\ldots, \mathscr{U}_r \subset \mathbb{R}^{2n}$ and positive numbers $0=a_0<a_1<\cdots < a_{r}=1$ satisfying the following conditions:
    \begin{itemize}
        \item [(i)] The subspaces $\mathscr{U}_1,\ldots, \mathscr{U}_r$ are invariant under both $JA$ and $JB$.
        \item [(ii)] Each subspace $\mathscr{U}_\ell$ for $\ell = 1,\ldots, r$ has a symplectic basis consisting of symplectic eigenvector pairs of $B$ corresponding to the symplectic eigenvalues $d_1(B),\ldots, d_k(B)$. 
        \item [(iii)] \sloppy For each $\ell \in \{1,\ldots, r\}$, there exists a non-trivial subinterval $[b_{\ell}, c_{\ell}] \subset [a_{\ell-1}, a_\ell]$ such that for all $t \in [b_{\ell}, c_{\ell}]$, the symplectic eigenvalues of $A+tB$ associated with $\mathscr{U}_\ell$ are $d_{i_1}(A+tB), \ldots, d_{i_k}(A+tB)$.
    \end{itemize}
\end{theorem}
%----------------------------------------------------------------------------------------------------------------------%
\begin{proof}
    By Theorem~4.6 of \cite{jm}, there exist analytic maps $\tilde{d}_1,\ldots, \tilde{d}_n:[0,1]\to \mathbb{R}$ such that $\tilde{d}_1(t),\ldots, \tilde{d}_n(t)$ are the $n$ symplectic eigenvalues (not necessarily in increasing order) of $A+tB$ for all $t \in [0,1]$. 
    It follows from the elementary theory of analytic functions \cite[Chapter~1]{krantz2002primer} that there exist finitely many numbers $0=a_0 <a_1<\cdots < a_r=1$ and $K \in \{1,\ldots, n\}$ such that for all $t \in [0,1]\backslash \{a_0,a_1,\ldots, a_r\}$, $A+tB$ has exactly $K$ distinct symplectic eigenvalues; these $K$ distinct symplectic eigenvalues can be assumed to be $\tilde{d}_1(t),\ldots, \tilde{d}_K(t)$ after some reordering, and each $\tilde{d}_i(t)$ repeats a fixed number of times $m_i$ for $i=1,\ldots, K$ so that $m_1+\cdots + m_K=n$. 
    We can thus arrange the symplectic eigenvalues of $A+tB$ for $t \in [0,1]\backslash \{a_0,a_1,\ldots, a_r\}$ as
    \begin{multline}
     d_{1}(A+tB) =\cdots = d_{m_1}(A+tB)< d_{m_1+1}(A+tB)=\cdots = d_{m_1+m_2}(A+tB)\\
        < \cdots < 
            d_{m_1+\cdots +m_{K-1}+1}(A+tB) = \cdots = d_{m_1+\cdots +m_{K}}(A+tB) \label{eq:analytic-ordered-sym-eigenvalues}
    \end{multline} 
    so that, given any $\ell \in \{1,\ldots,r\}$ and $i \in \{1,\ldots,n\}$, the restriction of the symplectic eigenvalue map
    $[0,1] \ni t \mapsto d_i(t) \coloneqq d_i(A+tB)$ to $(a_{\ell-1}, a_\ell)$ is equal to one of the analytical maps $\tilde{d}_1,\ldots, \tilde{d}_K$. 
    Also, without loss of generality, it follows from Theorem~4.7 of \cite{jm} that there exist maps $u_1,\ldots, u_n, v_1,\ldots, v_n : [0,1] \to \mathbb{R}^{2n}$ that are analytic on each open interval $(a_{\ell-1}, a_{\ell})$ for $\ell=1,\ldots, r$ such that $\{u_1(t),\ldots, u_n(t), v_1(t),\ldots, v_n(t)\}$ is a symplectic eigenbasis of $A+tB$ for all $t \in [0,1]$. 
    \sloppy 
    Set $M(t) \coloneqq \left[ u_{i_1}(t), \ldots, u_{i_k}(t), v_{i_1}(t),\ldots,v_{i_k}(t) \right] \in \operatorname{Sp}(2n, 2k)$ for all $t \in [0,1]$.

    Let us assume that the equality holds in \eqref{eq:sym-lidskii-inequality}.
    For every $\ell \in \{1,\ldots, r\}$, let us fix $t_\ell \in (a_{\ell-1}, a_{\ell})$. 
    Define $\mathscr{U}_\ell \coloneqq \operatorname{Ran}(M(t_\ell))$ for $\ell=1,\ldots,r$. 
    Observe that each $\mathscr{U}_\ell$ is a symplectic subspace of $\mathbb{R}^{2n}$ and is invariant under $J(A+t_\ell B)$. 
    This follows simply from \cite[Lemma~2.2]{jm} and the fact that the columns of $M(t_\ell)$ form symplectic eigenvector pairs of $A+t_\ell B$ corresponding to the symplectic eigenvalues $d_{i_1}(A+t_\ell B),\ldots, d_{i_k}(A+t_\ell B)$.
    The following arguments show that $\mathscr{U}_\ell$ is also an invariant subspace of $JB$.
    Recall from \eqref{eq_v1:first-equality-condition-5} that
        \begin{align}
            \dfrac{1}{2}\operatorname{Tr}\!\left[M(t_\ell)^T B M(t_\ell) \right] 
        		&= \sum_{j=1}^k d_j(B) \label{eq:first-equality-condition-5}.
        \end{align}
    By Theorem~4.6~(ii) of \cite{son2021symplectic}, there exists $U \in \operatorname{OrSp}(2k)$ such that the columns of $M(t_\ell) U$ consist of symplectic eigenvector pairs of $B$ corresponding to the symplectic eigenvalues $d_1(B),\ldots,d_k(B)$.
    Also, we have $\operatorname{Ran}(M(t_\ell) U)=\operatorname{Ran}(M(t_\ell))=\mathscr{U}_\ell$. 
    This means that the columns of $M(t_\ell) U$ form a desired symplectic basis of  $\mathscr{U}_\ell$ consisting of symplectic eigenvector pairs of $B$ corresponding to the symplectic eigenvalues $d_1(B),\ldots,d_k(B)$. 
    Consequently, $\mathscr{U}_\ell$ is also an invariant subspace of $JB$ \cite[Lemma~2.2]{jm}. 
    The fact that $\mathscr{U}_\ell$ is an invariant subspace of both $J(A+t_\ell B)$ and $JB$ implies that $\mathscr{U}_\ell$ is also invariant under $JA$.
    We have thus shown that the symplectic subspaces $\mathscr{U}_1,\ldots, \mathscr{U}_r$ satisfy the statements $(i)$ and $(ii)$ of the proposition. In the remaining part of the proof, we show that the statement $(iii)$ also holds.
    
    Since each $\mathscr{U}_\ell$ is invariant under both $JA$ and $JB$, it is also invariant under $J(A+tB)$ for all $t \in  [a_{\ell-1}, a_\ell]$. 
    Let $\gamma_{1,\ell}(t)\leq \cdots \leq \gamma_{k,\ell}(t)$ be the symplectic eigenvalues of $A+tB$ associated with $\mathscr{U}_\ell$ for $t \in  [a_{\ell-1}, a_\ell]$. 
    By Proposition~\ref{prop:piecewise-analyticity-sym-subspace}, we know that the maps $\gamma_{1,\ell},\ldots, \gamma_{k,\ell}$ are piecewise analytic on $[a_{\ell-1}, a_\ell]$.
    Thus, there exists a non-trivial subinterval $[b_\ell, c_\ell] \subset [a_{\ell-1}, a_\ell]$ containing $t_\ell$ such that $\gamma_{1,\ell},\ldots, \gamma_{k,\ell}$ are analytic on $(b_\ell, c_\ell)$.
    Also, recall that the symplectic eigenvalue maps $d_1,\ldots, d_n$ are analytic on $(b_\ell, c_\ell)$.
    By following similar arguments as given in the second paragraph of the proof of Proposition~\ref{prop:piecewise-analyticity-sym-subspace}, we get that for all $\ell \in \{1,\ldots,r\}$, there exist $k$-distinct indices $1\leq \sigma_1 < \cdots < \sigma_k \leq n$ such that $\gamma_{i,\ell}(t)=d_{\sigma_i}(A+tB)$ for all $t \in (b_{\ell}, c_\ell)$.
    We know that  $\gamma_{j,\ell}(t_\ell)=d_{i_j}(A+t_\ell B)$, because the symplectic eigenvalues of $A+t_\ell B$ associated with $\mathscr{U}_\ell$ are $d_{i_1}(A+t_\ell B), \ldots, d_{i_k}(A+t_\ell B)$. 
    In view of \eqref{eq:analytic-ordered-sym-eigenvalues} and continuity of symplectic eigenvalues \cite{bhatia2015symplectic}, we thus have $\gamma_{j,\ell}(t)=d_{i_j}(A+tB)$ for $t \in [b_{\ell}, c_\ell]$ and $j=1,\ldots, k$.  
    This proves statement $(iii)$ of the proposition.
\end{proof}
%----------------------------------------------------------------------------------------------------------------------%

%%%%%%%%%%%%%%%%%%%%%%%%%%%%%%%%%%%%%%%%%%%%%%%%%%%%%%%

\section{Majorization in symplectic Schur--Horn weak supermajorization}\label{sec:schur--horn}
    \sloppy Given any vector $x$ in $\mathbb{R}^n$, denote the entries of $x$ in the ascending order by $x^{\uparrow}_1 \leq  \cdots \leq x^{\uparrow}_n$. 
    Let $x,y$ be two vectors in $\mathbb{R}^n$. We say that $x$ is \textit{ weakly supermajorized} by $y$, written as $x \prec^w y$, if 
    \begin{align}\label{eq:supmajorization}
    \sum_{i=1}^k x^{\uparrow}_i \geq \sum_{i=1}^k y^{\uparrow}_i, \quad \text{for } k=1,\ldots, n.
    \end{align}
% Then $x$ is said to be \textit{ weakly submajorized} by $y$, written as $x\prec_w y$, if
% \begin{align}\label{eq:submajorization}
%     \sum_{i=1}^k x^{\downarrow}_i \leq \sum_{i=1}^k y^{\downarrow}_i, \quad \text{for } 1 \leq k \leq n.
% \end{align}
    In addition, if the equality in \eqref{eq:supmajorization} holds for $k=n$,  then $x$ is said to be \textit{ majorized} by $y$ and is written as $x \prec y$. 
% It is easy to prove that $x$ is majorized by $y$ if and only if $x$ is both, weakly submajorized as well as weakly supermajorized by $y$. 
    See \cite{marshall1979inequalities} for a comprehensive theory of majorization.
    An $n \times n$ real matrix $E$ is said to be a doubly stochastic matrix if its $(i,j)$th entries $E_{ij}$ are non-negative for $i,j=1,\ldots, n$ such that
    \begin{align}
        \sum_{j=1}^n E_{ij}=1, \quad \text{for } i=1, \ldots, n, \\
        \sum_{i=1}^n E_{ij}=1, \quad \text{for } j=1,\ldots, n.
    \end{align}
    An $n \times n$ real matrix $F$ is said to be a doubly superstochastic matrix 
% if its $(i,j)$th entries $Q_{ij}$ are non-negative for $1\leq i,j \leq n$ such that
% \begin{align}
%     \sum_{j=1}^n Q_{ij} \geq 1, \quad \text{for } 1 \leq i \leq n, \\
%     \sum_{i=1}^n Q_{ij} \geq 1, \quad \text{for } 1 \leq j \leq n.
% \end{align}
% Another characterization of doubly superstochastic matrices is as follows: $Q$ is a doubly superstochastic matrix if and only 
    if there exists a doubly stochastic matrix $E$ such that $E_{ij} \leq F_{ij}$ for all $i,j=1,\ldots, n$.
%See Theorem~1 of \cite{Bhandari1985}. 
    We recall the following known fundamental result in the theory of majorization that will be useful in the proof of the main result of the section. See \cite[Theorem~1.3]{ando1989}.
\begin{lemma}\label{lemma:doubly-stochastic}
    Given $x, y \in \mathbb{R}^n$, we have $x \prec y$ if and only if there exists an $n \times n$ doubly stochastic matrix $E$ such that $x=Ey$.
\end{lemma}
%and \cite[I.1.A.5]{marshall1979inequalities}
% \begin{lemma}\label{lemma:doubly-superstochastic}
%    Given $x, y \in \mathbb{R}^n$, we have $x \prec^w y$ if and only if there exists an $n \times n$ doubly superstochastic matrix $F$ such that $x=Fy$.
% \end{lemma}

    The classic Schur--Horn theorem gives a  relationship between the diagonal elements and the eigenvalues of a Hermitian matrix. 
    Let $X$ be an $n \times n$ Hermitian matrix. 
    Consider two real $n$-vectors  $\Delta(X)$ and $\lambda(X)$ whose entries are given by the diagonal entries and the eigenvalues of $X$, respectively.
    Schur \cite{schur1923uber} showed that the majorization relation $\Delta(X) \prec \lambda(X)$ always holds.
    The converse of Schur's result, proved by Horn \cite{horn1954doubly}, is also true: if $x,y$ are two vectors in $\mathbb{R}^n$ such that $x \prec y$, then there exists an $n\times n$ real symmetric matrix $X$ such that $\Delta(X)=x$ and $\lambda(X)=y$. See Section~9.B of \cite{marshall1979inequalities}.

    A symplectic version of the Schur--Horn theorem \cite[Theorem~3]{bhatia2020schur} states the following. Let $A$ be any $2n \times 2n$ real positive definite matrix in the block form
    \begin{align}
    A = \begin{pmatrix}
        A_{11} & A_{12} \\
        A_{12}^T & A_{22}
    \end{pmatrix},
    \end{align}
    where each block has size $n \times n$. The following weak supermajorization holds 
    \begin{align}\label{eqn18}
        \Delta_c(A) \prec^w d_s(A),
    \end{align}
    where $\Delta_c(A)$ and $d_s(A)$ are vectors of size $n$ given by
    \begin{align}
    \Delta_c(A) &\coloneqq \dfrac{1}{2} \left[\Delta(A_{11})+\Delta(A_{22})\right],\\
    d_s(A) &\coloneqq (d_1(A),\ldots, d_n(A)).
    \end{align}
    Conversely, if $x,y \in \mathbb{R}^{n}$ are vectors with positive entries such that $x \prec^w y$ then there exists $A \in \mathbb{P}_{2n}(\mathbb{R})$ such that $\Delta_c(A)=x$ and $d_s(A)=y$. 
    
%----------------------------------------------------------------------------------------------------------------------%
\begin{remark}\label{rem:proof_correction}
    There is a typo in the proof of the aforementioned symplectic Schur--Horn theorem (Theorem~3) given in \cite{bhatia2020schur}. 
    The proof works when the diagonal entries of the $2n \times 2n$ matrix $M_{\alpha}$ are replaced with $\sqrt{\alpha_1}, \ldots, \sqrt{\alpha_n}, \sqrt{\alpha_1^{-1}}, \ldots, \sqrt{\alpha_n^{-1}}$, where $\alpha=({\alpha_1}, \ldots, {\alpha_n})$ is an $n$-vector with positive entries.
\end{remark}
%----------------------------------------------------------------------------------------------------------------------%

%----------------------------------------------------------------------------------------------------------------------%
\begin{remark}\label{remark}
        There are three other known symplectic analogs of the Schur--Horn theorem. These Schur--Horn type theorems are given by replacing $\Delta_c(A)$ in \eqref{eqn18} with the following vectors associated with the diagonal of $A$:
    \begin{align}
     \Delta_s(A) &\coloneqq \left[\Delta(A_{11}) \Delta(A_{22}) \right]^{1/2}, \\
     \Delta_w(A) &\coloneqq 2^{-1/2}\left[\Delta(A_{11})^2  + \Delta(A_{22})^2\right]^{1/2}, \\
    \Delta_h(A) &\coloneqq 2^{-1/2}\left[\Delta(A_{11})^2 + \Delta(A_{22})^2 + 2 \Delta(A_{12})^2 \right]^{1/2},
    \end{align}
    where the product, squares, and square roots of the vectors are taken entry-wise. See \cite{bhatia_jain_2021, huang2023}.
    Recently, necessary and sufficient conditions for the majorization to hold in these symplectic Schur--Horn theorems were established in \cite{huang_mishra}. 
\end{remark}
%----------------------------------------------------------------------------------------------------------------------%

    In the following theorem, we provide exact description of positive definite matrices saturating the weak supermajorization in \eqref{eqn18} by majorization. 

%----------------------------------------------------------------------------------------------------------------------%
\begin{theorem}\label{thm:sym_schur_horn}
    Let $A \in \mathbb{P}_{2n}(\mathbb{R})$ and $D$ be the $n \times n$ diagonal matrix with diagonal entries $d_1(A),\ldots, d_n(A)$. 
    We have  $\Delta_c(A) \prec d_s(A)$ if and only if $A=N (D \oplus D)N^T$ for some $2n \times 2n$ orthosymplectic matrix $N$.
\end{theorem}
%----------------------------------------------------------------------------------------------------------------------%
%----------------------------------------------------------------------------------------------------------------------%
\begin{proof}
% It is known that $\widetilde{N}$ is a doubly superstochastic matrix \cite{bhatia2015symplectic}. The weak supermajorization \eqref{eqn18} thus holds, which follows from the equivalent condition discussed in the beginning of the section.
    By Williamson's theorem, there exists a symplectic matrix $M \in \operatorname{Sp}(2n)$ such that $A=M^{-T}(D \oplus D)(M^{-T})^T$. 
    Set $N \coloneqq M^{-T}$, and write this matrix in the block form
    \begin{align}
    N =   
        \begin{pmatrix}
            	P & Q \\
            	R & S
        \end{pmatrix},
    \end{align}
    where $P,Q,R,S$ are $n \times n$ blocks with $(i,j)$th entries $p_{ij}, q_{ij}, r_{ij}, s_{ij}$ respectively. 
    We thus get
    \begin{align}
    A = N(D \oplus D)N^T  
    	&= 
	\begin{pmatrix}
        		PDP^T+ QDQ^T & PDR^T+QDS^T \\
        		RDP^T+SDQ^T & RDR^T + SDS^T
    	\end{pmatrix}.
    \end{align}
    This implies
    \begin{align}
    \Delta_c(A) 
        &= \dfrac{1}{2} \left[\Delta(PDP^T)+\Delta(QDQ^T)+\Delta(RDR^T)+\Delta(SDS^T) \right] \\
        &= \dfrac{1}{2} \left[ (P\circ P) d_s(A)+(Q\circ Q) d_s(A)+(R\circ R) d_s(A)+(S\circ S) d_s(A) \right] \\
        &=\dfrac{1}{2} \left[ P\circ P+Q\circ Q+R\circ R+S\circ S  \right]d_s(A)\\
        &=\widetilde{N} d_s(A), \label{am_sym_diagonal_dss}
    \end{align}
    where $\circ$ denotes the Hadamard (entry-wise) product of matrices, and $\widetilde{N}$ is the $n\times n$ matrix with $(i,j)$th entry given by
    \begin{align}\label{eq:dss-associated-with-N-tilde}
        \widetilde{N}_{ij}=\dfrac{p_{ij}^2+q_{ij}^2+r_{ij}^2+ s_{ij}^2}{2}, \qquad i,j=1,\ldots,n.
    \end{align}
    Since the symplectic group is closed under matrix inverse and transpose, we have $N \in\operatorname{Sp}(2n)$. 
    It then follows by Theorem~6 of \cite{bhatia2015symplectic} that $\widetilde{N}$ is a doubly superstochastic matrix. 
    By the definition of doubly superstochastic matrices, there exists an $n\times n$ doubly stochastic matrix $E$ whose $(i,j)$th entry satisfies $E_{ij} \leq \widetilde{N}_{ij}$ for all $i,j=1,\ldots, n$. 
    Also, from \eqref{am_sym_diagonal_dss}, we have
    \begin{align}\label{eq:sym-ds-dss}
    (\widetilde{N}-E)d_s(A)= \Delta_c(A) - E d_s(A).
    \end{align}

    Assume that $\Delta_c(A) \prec d_s(A)$.
    By equating the sums of the elements of the vectors on both sides of \eqref{eq:sym-ds-dss}, we thus get
    \begin{align}
        \sum_{i=1}^n \sum_{j=1}^n (\widetilde{N}_{ij}-E_{ij}) d_j(A)=0.
    \end{align}
    Since $\widetilde{N}_{ij}-E_{ij} \geq 0$ and $d_j(A)>0$ for all $i,j \in \{1,\ldots, n\}$, we get $\widetilde{N}_{ij}=E_{ij}$ for all $i,j=1,\ldots,n$; i.e., $\widetilde{N}$ is a doubly stochastic matrix. 
    It then follows by Theorem~6 of \cite{bhatia2015symplectic} that $N$ is an orthosymplectic matrix.

    Conversely, assume that $A=N (D \oplus D)N^T$, where $N$ is a $2n \times 2n$ orthosymplectic matrix. 
    The matrix $N$ being orthosymplectic implies that the matrix $\widetilde{N}$ given by \eqref{eq:dss-associated-with-N-tilde} is doubly stochastic \cite[Theorem~6]{bhatia2015symplectic}. 
    It then follows from \eqref{am_sym_diagonal_dss} and Lemma~\ref{lemma:doubly-stochastic} that $\Delta_c(A) \prec d_s(A)$.
\end{proof}
%----------------------------------------------------------------------------------------------------------------------%

%%%%%%%%%%%%%%%%%%%%%%%%%%%%%%%%%%%%%%%%%%%%%%%%%%%%%%%
%%%%%%%%%%%%%%%%%%%%%%%%%%%%%%%%%%%%%%%%%%%%%%%%%%%%%%%

\section{Summary}
    We established necessary and sufficient conditions for equality in various symplectic eigenvalue inequalities such as symplectic Weyl's inequalities, symplectic Lidskii's inequalities, and a symplectic Schur--Horn weak supermajorization inequality.
    Another interesting consequence of our analysis is that a symplectic analog of the well-known fact about Hermitian matrices and their invariant subspaces: if $\mathscr{U}$ is a symplectic subspace invariant under $JA$ then it has a symplectic basis consisting of symplectic eigenvector pairs of $A$.

%------
% Insert acknowledgments and information
% regarding funding at the end of the last
% section, i.e., right before the bibliography.
%------

\begin{ack}
    I am grateful to Prof.~Tanvi Jain for the encouragement and helpful suggestions that made this project possible. 
    I  thank Prof.~Shmuel Friedland for pointing me to his book which was helpful in finishing the paper. 
    I am indebted to Prof. Mark Wilde for his immense support and kindness. 
    I thank Komal Malik for her help in resolving a technical issue in the paper.\end{ack}

\begin{funding}
    This research was supported by the National Science Foundation under
Grant No.~2304816
\end{funding}

%------
% Insert the bibliography.
%------

\end{document}